\documentclass[12pt]{amsart}
\usepackage{amsmath,amsthm,amsfonts,amssymb,stmaryrd,mathrsfs,enumitem}
\usepackage[utf8]{inputenc}
\usepackage{mathrsfs}
\usepackage{amssymb,amsmath,amscd}
\usepackage{setspace}
\usepackage[mathscr]{euscript}
\usepackage{color}
\usepackage[all]{xy}
\usepackage{graphics}
\usepackage[english]{babel}
\usepackage{url}

\usepackage[colorlinks,linkcolor=blue,citecolor=blue]{hyperref}

\usepackage{latexsym}
\usepackage{tikz-cd}
\usepackage{fullpage}
\usepackage{a4,color,palatino,fancyhdr}
\usepackage{graphicx}
\usepackage{float}  
\usepackage[small, bf, margin=90pt, tableposition=bottom]{caption}
\RequirePackage{amssymb}
\RequirePackage[T1]{fontenc}
\usepackage{amscd}
\usepackage[all]{xy}
\newtheorem{thm}{Theorem}[section]
\newtheorem{prop}[thm]{Proposition}

\newtheorem{lem}[thm]{Lemma}
\newtheorem{coro}[thm]{Corollary}

\newtheorem{example}[thm]{Example}

\newtheorem{rem}[thm]{Remark}

\numberwithin{equation}{section}

\def\Z{\mathbb Z}

\def\Q{\mathbb Q}

\def\D{\mathrm{Diff}_{\partial}}
\title{A short note on $\pi_1\D( D^{4k})$ for $k\geq 3$}

\author{Wei Wang}
\address{Department of Mathematics and Computational Science, Shanghai Ocean University, Shanghai 201306, China}
\email{weiwang@amss.ac.cn}

\date{}

\begin{document}

\maketitle

\begin{abstract}
Let $\D(D^{n})$ be the topological group of diffeomorphisms of $D^{n}$ which agree with the identity near the boundary. In this short note, we compute the fundamental group $\pi_1 \D(D^{4k})$ for $k\geq 3$.
\end{abstract}

\section{Introduction}
Let $\D(D^n)$ be the topological group of diffeomorphisms of a disc $D^n$ which agree with the identity near the boundary of $D^n$ endowed with the Whitney topology (e.g. \cite[Chapter 2]{Hirsch}, \cite[Chapter 3-6]{Kupers}). The homotopy type of $\D(D^n)$ is a very important object in algebraic and geometric topology.

When $n=1,2,3$, $\D(D^n)$ is contractible  (e.g. \cite{Hat2}). Watanabe \cite{Watanabe1} disproved the  4-dimensional Smale
conjecture that $\D(D^4)$ is contractible. When $n\geq 5$, $\D(D^n)$ is not contractible  \cite{Hat2} and many computations of $\pi_i\D(D^n)$ have been done. Burghelea and Lashof \cite{BL,BL2} constructed many torsion elements in $\pi_i\D(D^n)$.  
Farrell and Hsiang computed $\pi_i \D(D^n)\otimes \Q$ in the stable concordance range. Recent breakthroughs on $\pi_i \D(D^n)\otimes \Q$ have been made by Galatius, Krannich, Kupers, Randal-Williams, Watanabe, Weiss and other people. We refer to Randal-Williams's ICM talk \cite{RW} for more results and references.

In this short note, we will be concerned with the fundamental group $\pi_{1}\D(D^{4k})$ of $\D(D^{4k})$.
Our main result is:
\begin{thm}\label{main thm}
	When $k\geq 3$, $\pi_1\D(D^{4k})\cong \Theta_{4k+2}$, where $\Theta_{4k+2}$ denotes the group of exotic spheres of dimension $4k+2$.\\
\end{thm}
\begin{rem}
 According to \cite[Remark 1.2]{Watanabe1}, $\pi_1\D(D^{4})\otimes \Q \cong \pi_2 \mathrm{B}\D(D^{4})\otimes \Q \neq\Theta_6\otimes \Q =0$. It follows that Theorem \ref{main thm} does not hold when $k=1$. The case when $k=2$ is not known, see Remark \ref{Remark1} for more details.
\end{rem}

As a consequence of Theorem \ref{main thm}, let $\mathrm{Diff}(S^n)$ be the topological group of diffeomorphisms of $S^n$. We have
\begin{coro}
When $k\geq 3$, 		$\pi_1\mathrm{Diff}(S^{4k})\cong \Z_2\oplus \Theta_{4k+2}$.\\
\end{coro}
\begin{proof}
Similar to the proof of Lemma 1.1.5 in \cite{ABK}, one can show $\mathrm{Diff}(S^n)$ is homotopy equivalent as a topological space to the product of the orthogonal group $\mathrm{O}_{n+1}$ and $\D(D^n)$. Then we have
\[
\pi_1\mathrm{Diff}(S^{4k})\cong\pi_1(\mathrm{O}_{4k+1}\times \D(D^{4k})) \cong \Z_2\oplus \Theta_{4k+2}.
\]
\end{proof}
In the process of computations, we can obtain information about some Gromoll filtration groups (see \cite[Section 1]{Crowley} for the definition). Precisely, we have
\begin{thm}\label{mainthm2}
Let $\Gamma_3^{4k-1}$ be the 3-rd Gromoll filtration group in $\Theta_{4k-1}$ and let $bP_{4k}$ be the subgroup of $\Theta_{4k-1}$ consists of homotopy spheres which bound parallelizable $4k$-manifolds. When $k\geq 4$, we have
\begin{itemize}
	\item[(a)]  $\Theta_{4k-1}/\Gamma_3^{4k-1}\cong \Z_2$.
	\item[(b)] There is an extension of abelian groups 	
	\[
	\begin{CD}
0 @>>> 2bP_{4k}@>>>  \Gamma_3^{4k-1} @>>>\Theta_{4k-1}/bP_{4k}@>>> 0,\\
	\end{CD}
	\]
	where $2bP_{4k}$ denotes the subgroup of $bP_{4k}$ with $bP_{4k}/2bP_{4k}\cong \Z_2$.
	\item[(c)] The extension class $[\Gamma_{3}^{4k-1}]$ belongs to the kernel of the homomorphism $$\mathrm{Ext}(\Theta_{4k-1}/bP_{4k},2bP_{4k})\longrightarrow \mathrm{Ext}(\Theta_{4k-1}/bP_{4k},bP_{4k}),$$ 
	which is induced by the inclusion $2bP_{4k}\longrightarrow bP_{4k}$.
\end{itemize}


\end{thm}

\begin{rem}
	Gromoll filtration groups have been studied by many people (e.g. \cite{ABK,Crowley,Crowley2,Weiss}). Complete known results have been collected in the appendices in \cite{Crowley,Crowley2}.
\end{rem}

We also have an application to the fundamental group of diffeomorphism groups of certain manifolds. Let
$M^{4k}$ be a compact smooth manifold, possibly with boundary, whose dimension is $4k$, $k\geq 3$.
Let $\D(M^{4k})$ be the group of diffeomorphisms of $M^{4k}$ which are identity near the boundary and let $\widetilde{\mathrm{Diff}}_{\partial}(M^{4k})$ be the block diffeomorphism group of $M^{4k}$ (see e.g. \cite[\S 1, \S 2]{Krannich} for the definition). We have

\begin{thm}\label{mainthm3}
When $M^{4k}$ is 3-connected with dimension $4k\geq 12$, the inclusion induces an isomorphism
\[
\pi_{1}\D(M^{4k})\cong \pi_1 \widetilde{\mathrm{Diff}}_{\partial}(M^{4k}).
\]
\end{thm}

\

This note is organized as follows. In Section 2 we consider the exact sequence of homotopy groups $\pi_r\D(D^n)\rightarrow \pi_r \widetilde{\mathrm{Diff}}_{\partial}(D^n)\rightarrow \pi_r( \frac{\widetilde{\mathrm{Diff}}_{\partial}(D^n)}{\D(D^n)})$ and compute  $\pi_2( \frac{\widetilde{\mathrm{Diff}}_{\partial}(D^n)}{\D(D^n)})$ for $n\geq 9$. Section 3 discusses the Gromoll filtration group $\Gamma_3^{4k-1}$ for $k\geq 4$. In Section 4 we prove Theorem \ref{main thm} and Theorem \ref{mainthm3}.


\section{$\pi_2(\frac{\widetilde{\mathrm{Diff}}_{\partial}(D^n)}{\D(D^n)} )$ for $n\geq 9$ }
In this section, we assume $n\geq 9$ unless mentioned otherwise.
Let $\widetilde{\mathrm{Diff}}_{\partial}(D^n)$ be the group of block diffeomorphisms of $D^n$. Let $i\colon\D(D^n)\rightarrow\widetilde{\mathrm{Diff}}_{\partial}(D^n)$ be the natural inclusion, the induced homomorphism $i_0\colon\pi_0\D(D^n) \longrightarrow \pi_0\widetilde{\mathrm{Diff}}_{\partial}(D^n)$ is an isomorphism due to Cerf's pseudoisotopy theorem \cite{Cerf} and Smale's $h$-cobordism theorem \cite{Smale}. 
We consider the following long exact sequence of homotopy groups induced by the fiber sequence $\D(D^n)\rightarrow \widetilde{\mathrm{Diff}}_{\partial}(D^n) \rightarrow \frac{\widetilde{\mathrm{Diff}}_{\partial}(D^n)}{\D(D^n)}$
\[
\begin{tikzcd}
		\cdots \rar  &\pi_2\D(D^n) \rar &  \pi_2\widetilde{\mathrm{Diff}}_{\partial}(D^n)
	\rar  \ar[draw=none]{d}[name=X, anchor=center]{}
	&\pi_2(\frac{\widetilde{\mathrm{Diff}}_{\partial}(D^n)}{\D(D^n)} ) \ar[rounded corners,
	to path={ -- ([xshift=2ex]\tikztostart.east)
		|- (X.center) \tikztonodes
		-| ([xshift=-2ex]\tikztotarget.west)
		-- (\tikztotarget)}]{dll}[at end]{\ } \\      
  &\pi_1\D(D^n) \rar &  \pi_1\widetilde{\mathrm{Diff}}_{\partial}(D^n)
	\rar  \ar[draw=none]{d}[name=X, anchor=center]{}
	&\pi_1(\frac{\widetilde{\mathrm{Diff}}_{\partial}(D^n)}{\D(D^n)}  ) \ar[rounded corners,
	to path={ -- ([xshift=2ex]\tikztostart.east)
		|- (X.center) \tikztonodes
		-| ([xshift=-2ex]\tikztotarget.west)
		-- (\tikztotarget)}]{dll}[at end]{\ } \\      
	&\pi_0\D(D^n) \rar["\cong"] &\pi_0\widetilde{\mathrm{Diff}}_{\partial}(D^n).
	 & 
\end{tikzcd}
\]
 It is known that $\pi_i\widetilde{\mathrm{Diff}}_{\partial}(D^n)\cong \pi_{0}\D(D^{n+i})\cong \Theta_{n+i+1}$, the group of exotic $(n+i+1)$-spheres (e.g. \cite[\S 2.3.2]{ABK}). In order to compute $\pi_1 \D(D^n)$, one should deal with the relative homotopy group $\pi_i(\frac{\widetilde{\mathrm{Diff}}_{\partial}(D^n)}{\D(D^n)})$ for $i=1,2$.

We will use the Hatcher spectral sequence \cite[Section 2]{Hat1} $$E^1_{p,q}=\pi_q\mathcal{C}(D^n\times [0,1]^p)\Longrightarrow \pi_{p+q+1}\Big({\frac{\widetilde{\mathrm{Diff}}_{\partial}(D^n)}{\D(D^n)}} \Big),$$ 
where $\mathcal{C}(D^n\times [0,1]^p)$ denotes the concordance group of $D^n\times [0,1]^p$ (\cite[Section 1]{Hat1}) which consists of diffeomorphisms of $D^n\times [0,1]^p\times [0,1]$ that fix each point of $D^n\times [0,1]^p\times 0$ and $\partial(D^n\times [0,1]^p)\times [0,1]$.

For $\pi_{1}(\frac{\widetilde{\mathrm{Diff}}_{\partial}(D^n)}{\D(D^n)} )$, by Cerf's pseudoisotopy theorem, $E_{0,0}^1=\pi_0\mathcal{C}(D^n\times [0,1]^p)=0$ when $n\geq 5$ and $p\geq 0$. It follows that $\pi_{1}(\frac{\widetilde{\mathrm{Diff}}_{\partial}(D^n)}{\D(D^n)} )=0$ when $n\geq 5$.

For $\pi_{2}(\frac{\widetilde{\mathrm{Diff}}_{\partial}(D^n)}{\D(D^n)} )$, it is known that $E^1_{1,0}=\pi_0\mathcal{C}(D^n\times [0,1])=0$ when $n\geq 4$. 
The only differential we should consider is
$$d^1\colon E^1_{1,1}=\pi_1\mathcal{C}(D^{n+1})\longrightarrow E^1_{0,1}=\pi_1\mathcal{C}(D^n).$$
\begin{lem}\label{rlemma}
The differential $d^1\colon E^1_{1,1}\longrightarrow E^1_{0,1}$ is trivial.
\end{lem}
\begin{proof}
Let $\sigma\colon\mathcal{C}(D^n)\rightarrow \mathcal{C}(D^n\times [0,1])$ be the stabilization map.
According to Igusa's stability theorem for $D^n$ (see \cite[Page 7, Proposition \& Page 10, Theorem]{Igusa}, the induced homomorphism
$$\sigma_*\colon\pi_1\mathcal{C}(D^{n})\rightarrow \pi_1 \mathcal{C}(D^{n+1})$$
is an isomorphism. It is also known that the stable group $\lim\limits_{n\rightarrow \infty}\pi_1\mathcal{C}(D^n)$ induced by $\sigma_{*}$ is isomorphic to $\Z_2$ (see Remark \ref{remofC1} for more details).

Let $[g]\in \pi_1\mathcal{C}(D^{n})$, by the description of $d^1$ (see e.g. \cite[Page 6]{Hat1}, \cite[Page 199]{Kupers}, we have
\[
d^1(\sigma_*[g])=[g]+[\overline{g}],
\] 
where $``-"$ is the involution defined on $\mathcal{C}(D^n)$ by reflecting $[0,1]$ from the midpoint and normalizing \cite[Page 6]{Hat1}.

Since $\pi_1\mathcal{C}(D^n)\cong \Z_2$, the involution on $\pi_1\mathcal{C}(D^n)$ is trivial and
$d^1(\sigma_*[g])=2[g]=0$. 
\end{proof}

By Lemma \ref{rlemma}, $E^{\infty}_{1,0}=0, E^{\infty}_{0,1}\cong \pi_1\mathcal{C}(D^n)\cong \Z_2$. Therefore we have \begin{prop}\label{mainProp}
	$\pi_2(\frac{\widetilde{\mathrm{Diff}}_{\partial}(D^n)}{\D(D^n)} ) \cong\Z_2$.
\end{prop}
As a quick application, we calculate two special examples.


\begin{example}
	 $\pi_1 \D(D^{10})\cong \Z_2$.
	 
Note that the homomorphism $\pi_2\widetilde{\mathrm{Diff}}_{\partial}(D^{10})\cong \Theta_{13}\cong \Z_3 \longrightarrow 	\pi_2(\frac{\widetilde{\mathrm{Diff}}_{\partial}(D^{10})}{\D(D^{10})} )$ is trivial and $\pi_1\widetilde{\mathrm{Diff}}_{\partial}(D^{10})\cong \Theta_{12}=0$. It follows that $\pi_1 \D(D^{10})\cong \pi_2(\frac{\widetilde{\mathrm{Diff}}_{\partial}(D^{10})}{\D(D^{10})} )\cong \Z_2$.

\end{example}

\begin{example}
	$\pi_1 \D(D^{27})\cong \Z_2\oplus \Z_3$.
	
According to \cite[Table 1]{IsakenWangXu}, $\Theta_{29}\cong\Z_3$ and $\Theta_{30}\cong \Z_3$. The homomorphism $\pi_2\widetilde{\mathrm{Diff}}_{\partial}(D^{27})\cong \Theta_{30}\cong \Z_3 \longrightarrow 	\pi_2(\frac{\widetilde{\mathrm{Diff}}_{\partial}(D^{27})}{\D(D^{27})} )$ is trivial. One has the following exact sequence of abelian groups
\[
0\longrightarrow \pi_2\Big(\frac{\widetilde{\mathrm{Diff}}_{\partial}(D^{27})}{\D(D^{27})} \Big) \cong\Z_2\longrightarrow \pi_1\D(D^{27})\longrightarrow  \pi_1\widetilde{\mathrm{Diff}}_{\partial}(D^{27})\cong \Z_3 \longrightarrow 0.
\]
This sequence splits and $\pi_1 \D(D^{27})\cong \Z_2\oplus \Z_3$. 
	
\end{example}

\begin{rem}\label{remofC1}
Let $\mathscr{C}(D^n)=\mathrm{hocolim}_{s}\mathcal{C}(D^n\times [0,1]^s)$ be the homotopy colimit induced by the stabilization map $\sigma\colon\mathcal{C}(D^n\times [0,1]^s)\longrightarrow \mathcal{C}(D^n\times [0,1]^{s+1})$. By  Igusa's stability theorem, $ \pi_1\mathcal{C}(D^n) \cong  \pi_1\mathscr{C}(D^n)$.

On the other hand, $ \pi_1\mathscr{C}(D^n)\cong \pi_3 {Wh}^{\mathrm{Diff}}(D^n)\cong  \Z_2$, where ${Wh}^{\mathrm{Diff}}(D^n)$ is the Whitehead space of $D^n$ (see e.g. \cite[Page 7, 8]{Hat2}, \cite[Page 178, 179]{Rognes1}, \cite[Page 915]{Rognes2}). 
\end{rem}

\begin{rem}\label{referee}
In Hatcher's survey \cite{Hat1}, his Corollary 3.3 says $ \pi_1\mathcal{C}(D^n)$ has order 4 ($n$ large). This is NOT correct.
As one of the referees mentioned, the reason is that the exact sequence in his Theorem 3.2 (cited from a never-published Igusa's
preprint) is incorrect, which is not necessarily surjective at the right. For the correct version of this sequence and the correct definition of the abelian group $\mathrm{Wh}_3(e)$, see Definition 2.6 and
Corollary 2.7 of \cite{Jahren}. With these corrections, the problems in Section 8 of \cite{Hat1} are resolved.
\end{rem}

\section{The Gromoll filtration group $\Gamma_3^{4k-1}$ for $k\geq 4$}
According to \cite[Corollary 2.3.3]{ABK}, the image of the homomorphism $\pi_2\D(D^n)\longrightarrow \pi_2\widetilde{\mathrm{Diff}}_{\partial}(D^n)$
is equal to the Gromoll filtration group $\Gamma^{n+3}_{3}$. When $n\geq 9$, by Proposition \ref{mainProp}, we have exact sequence
\[
\pi_2\D(D^n)\longrightarrow  \pi_2\widetilde{\mathrm{Diff}}_{\partial}(D^n) \longrightarrow \pi_2\Big(\frac{\widetilde{\mathrm{Diff}}_{\partial}(D^n)}{\D(D^n)} \Big) \cong\Z_2.
\]
Observe that the quotient group $\pi_2\widetilde{\mathrm{Diff}}_{\partial}(D^n)/\Gamma_3^{n+1}\cong \Theta_{n+3}/\Gamma_3^{n+3}$ is a subgroup of $\Z_2$.

\

\begin{proof}[Proof of Theorem \ref{mainthm2}]
	
 When $n=4k\geq 8$, according to \cite[Section 6.6]{WeissWilliams}, the generator of the finite cyclic group $bP_{4k}$ does not belong to $\Gamma_3^{4k-1}$. Combing the discussion above, the quotient group $\Theta_{4k-1}/\Gamma_3^{4k-1}\cong \Z_2$ when $k\geq 4$. Part (a) follows.

For Part (b), note that the subgroup $2bP_{4k}=\{2x \ |\ x\in bP_{4k}\}$ lies in $\Gamma_3^{4k-1}$ and one has the following exact sequences
\[
\begin{CD}
0 @>>> bP_{4k}@>>> \Theta_{4k-1} @>>> \Theta_{4k-1}/bP_{4k}@>>> 0\\
@.    @AAA        @AAA    @AA\cong A  @.\\
0 @>>> 2bP_{4k}@>>> \Gamma_3^{4k-1} @>>> \Gamma_3^{4k-1}/2bP_{4k}@>>> 0.\\
\end{CD}
\]
Here $\Gamma_3^{4k-1}/2bP_{4k}\cong \Theta_{4k-1}/bP_{4k}$ is due to the Snake Lemma. 

Recall that $\Theta_{4k-1}$ is an abelian group extension of $bP_{4k}$ by $\Theta_{4k-1}/bP_{4k}$ (e.g. \cite[Chapter 3]{Maclane}). Observe that the extension class $[\Theta_{4k-1}]\in \mathrm{Ext}(\Theta_{4k-1}/bP_{4k},bP_{4k})$ is the image of the extension class $[\Gamma_3^{4k-1}]$ through the homomorphism 
$$ \mathrm{Ext}(\Theta_{4k-1}/bP_{4k},2bP_{4k})\longrightarrow \mathrm{Ext}(\Theta_{4k-1}/bP_{4k},bP_{4k}).$$

According to \cite[Theorem 1.3]{Brumfiel}, the upper horizontal exact sequence
splits. Then $[\Theta_{4k-1}]=0\in \mathrm{Ext}(\Theta_{4k-1}/bP_{4k},bP_{4k})$ and Part (c) follows.
\end{proof}
\begin{example}
	$\Gamma_{3}^{27}\cong 2bP_{28}$, $\Gamma_{3}^{43}\cong 2bP_{44}$ and $\Gamma_{3}^{55}\cong 2bP_{56}\oplus \Z_3$.
	
According to \cite[Table 1]{IsakenWangXu}, $\Theta_{27}=bP_{28}$ , $\Theta_{43}=bP_{44}$ and $\Theta_{55}\cong bP_{56}\oplus \Z_3$. By Theorem \ref{mainthm2}-(a), one has 	$\Gamma_{3}^{27}\cong 2bP_{28}$ and $\Gamma_{3}^{43}\cong 2bP_{44}$. $\Theta_{55}/bP_{56}\cong \Z_3$ implies the homomorphism $\mathrm{Ext}(\Theta_{55}/bP_{56}, 2bP_{4k})\longrightarrow \mathrm{Ext}(\Theta_{55}/bP_{56}, bP_{4k})$ is injective. By Theorem \ref{mainthm2}-(c), the extension class $[\Gamma_{3}^{55}]$ is trivial and $\Gamma_{3}^{55}\cong 2bP_{56}\oplus \Z_3$.
\end{example}

\section{Proof of Theorem \ref{main thm} and Theorem \ref{mainthm3}} 
\begin{proof}[Proof of Theorem \ref{main thm}]
When $n=4k$ with $k\geq 3$,  $\pi_2\widetilde{\mathrm{Diff}}_{\partial}(D^n)/\Gamma_3^{n+1}\cong \Z_2$ due to Theorem \ref{mainthm2}-(a). Then the homomorphism $$\pi_2\widetilde{\mathrm{Diff}}_{\partial}(D^{4k}) \longrightarrow \pi_2\Big(\frac{\widetilde{\mathrm{Diff}}_{\partial}(D^{4k})}{\D(D^{4k})} \Big ) \cong\Z_2$$ 
is surjective and one has isomorphism
\[
\pi_1\D(D^{4k})\longrightarrow  \pi_1\widetilde{\mathrm{Diff}}_{\partial }(D^{4k}) \cong \pi_0\widetilde{\mathrm{Diff}}_{\partial}(D^{4k+1})\cong \Theta_{4k+2}.
\]
\end{proof}
\begin{rem}\label{Remark1}
In the proof of Proposition \ref{mainProp}, we need Igusa's stability theorem and the condition $n\geq 9$ is necessary. When $n=8$, the method in this note seems not very helpful to the calculation of the group $\pi_2(\frac{\widetilde{\mathrm{Diff}}_{\partial}(D^8)}{\D(D^8)})$.   
\end{rem}

At the end of this note, we would like to use the results above to prove Theorem \ref{mainthm3}.
\begin{proof}[Proof of Theorem \ref{mainthm3}]
Choose a disc $D^{4k}$ in the interior of $M^{4k}$ and one has group homomorphism
\[
\D(D^{4k})\rightarrow \D(M^{4k}) 
\]
\[
f\mapsto \hat{f},
\]
where $\hat{f}$ is the extension of $f$ with $\hat{f}|_{D^{4k}}=f$ and $\hat{f}|_{M^{4k}-D^{4k}}=id$.
Consider the following commutative diagram of exact sequences
\[
\begin{CD}
	\pi_2(\frac{\widetilde{\mathrm{Diff}}_{\partial}(D^{4k})}{\D(D^{4k})} ) @> trivial >> \pi_1\D(D^{4k}) @>>>   \pi_1\widetilde{\mathrm{Diff}}_{\partial}(D^{4k}) @>>>	\pi_1(\frac{\widetilde{\mathrm{Diff}}_{\partial}(D^{4k})}{\D(D^{4k})} )\\
	@V\cong VV   @VVV    @VVV  @VVV\\
		\pi_2(\frac{\widetilde{\mathrm{Diff}}_{\partial}(M^{4k})}{\D(M^{4k})} ) @>>> \pi_1\D(M^{4k}) @>>>   \pi_1\widetilde{\mathrm{Diff}}_{\partial}(M^{4k}) @>>>	\pi_1(\frac{\widetilde{\mathrm{Diff}}_{\partial}(M^{4k})}{\D(M^{4k})} ).\\
\end{CD}
\]
When $M^{4k}$ is 3-connected, by Morlet disjunction \cite[Corollary 3.2]{BLR}, $\pi_i(\frac{\widetilde{\mathrm{Diff}}_{\partial}(D^{4k})}{\D(D^{4k})} )\cong \pi_i(\frac{\widetilde{\mathrm{Diff}}_{\partial}(M^{4k})}{\D(M^{4k})} )$ for $1\leq i \leq 3$. 
$\pi_1(\frac{\widetilde{\mathrm{Diff}}_{\partial}(M^{4k})}{\D(M^{4k})})\cong  \pi_1(\frac{\widetilde{\mathrm{Diff}}_{\partial}(D^{4k})}{\D(D^{4k})} )=0$ implies $\pi_1\D(M^{4k}) \longrightarrow \pi_1\widetilde{\mathrm{Diff}}_{\partial}(M^{4k})$ is surjective.

According to Theorem \ref{main thm}, $\pi_1\D(D^{4k}) \cong   \pi_1\widetilde{\mathrm{Diff}}_{\partial}(D^{4k})$ implies the boundary homomorphism $\pi_2(\frac{\widetilde{\mathrm{Diff}}_{\partial}(D^{4k})}{\D(D^{4k})} ) \longrightarrow \pi_1\D(D^{4k})$ is trivial.
The isomorphism $\pi_2(\frac{\widetilde{\mathrm{Diff}}_{\partial}(M^{4k})}{\D(M^{4k})})\cong  \pi_2(\frac{\widetilde{\mathrm{Diff}}_{\partial}(D^{4k})}{\D(D^{4k})} )$ implies the boundary homomorphism
$\pi_2(\frac{\widetilde{\mathrm{Diff}}_{\partial}(M^{4k})}{\D(M^{4k})} ) \longrightarrow \pi_1\D(M^{4k})$ is also trivial, which shows that the homomorphism $\pi_1\D(M^{4k}) \longrightarrow \pi_1\widetilde{\mathrm{Diff}}_{\partial}(M^{4k})$ is injective.
\end{proof}

\section*{Acknowledgments}
The author would like to thank the anonymous referees for their valuable comments. In particular, Remark \ref{referee} is due to one of them. The author would also like to thank Yi Jiang, Zhi L\"{u}, Jianzhong Pan and Yang Su for various helpful discussions. Many thanks to Alexander Kupers for many suggestions on studying topology of diffeomorphism groups and providing many useful references.

\end{document}